\documentclass[11pt]{amsart}
\usepackage{amsmath,amssymb,amsbsy,amsfonts,amsthm,latexsym,amsopn,amstext,amsxtra,euscript,amscd,color,bm}                                               
\usepackage{mathrsfs}
\usepackage[right=3cm,left=3cm,top=2cm,bottom=2cm]{geometry}
\usepackage{fourier} 
\usepackage[np]{numprint}
\npdecimalsign{\ensuremath{.}}
                                                                                           
\usepackage[colorlinks,linkcolor=blue,anchorcolor=blue,citecolor=blue,backref=page]{hyperref}

\hypersetup{breaklinks=true}

\DeclareMathOperator{\id}{Id}
\DeclareMathOperator{\res}{Res}

\begin{document}

\def\ov#1{{\overline{#1}}} 
\def\un#1{{\underline{#1}}}
\def\wh#1{{\widehat{#1}}}
\def\wt#1{{\widetilde{#1}}}

\newcommand{\Ch}{{\operatorname{Ch}}}
\newcommand{\Elim}{{\operatorname{Elim}}}
\newcommand{\proj}{{\operatorname{proj}}}
\newcommand{\h}{{\operatorname{h}}}

\newcommand{\hh}{\mathrm{h}}
\newcommand{\aff}{\mathrm{aff}}
\newcommand{\Spec}{{\operatorname{Spec}}}
\newcommand{\Res}{{\operatorname{Res}}}
\newcommand{\Orb}{{\operatorname{Orb}}}

\renewcommand*{\backref}[1]{}
\renewcommand*{\backrefalt}[4]{%
    \ifcase #1 (Not cited.)%
    \or        (p.\,#2)%
    \else      (pp.\,#2)%
    \fi}
\def\lc{{\mathrm{lc}}}
\newcommand{\hcan}{{\operatorname{\wh h}}}

\newcommand{\hooklongrightarrow}{\lhook\joinrel\longrightarrow}

\newcommand{\bfa}{{\boldsymbol{a}}}
\newcommand{\bfb}{{\boldsymbol{b}}}
\newcommand{\bfc}{{\boldsymbol{c}}}
\newcommand{\bfd}{{\boldsymbol{d}}}
\newcommand{\bff}{{\boldsymbol{f}}}
\newcommand{\bfg}{{\boldsymbol{g}}}
\newcommand{\bfell}{{\boldsymbol{\ell}}}
\newcommand{\bfp}{{\boldsymbol{p}}}
\newcommand{\bfq}{{\boldsymbol{q}}}
\newcommand{\bfs}{{\boldsymbol{s}}}
\newcommand{\bft}{{\boldsymbol{t}}}
\newcommand{\bfu}{{\boldsymbol{u}}}
\newcommand{\bfv}{{\boldsymbol{v}}}
\newcommand{\bfw}{{\boldsymbol{w}}}
\newcommand{\bfx}{{\boldsymbol{x}}}
\newcommand{\bfy}{{\boldsymbol{y}}}
\newcommand{\bfz}{{\boldsymbol{z}}}

\newcommand{\bfA}{{\boldsymbol{A}}}
\newcommand{\bfF}{{\boldsymbol{F}}}
\newcommand{\bfG}{{\boldsymbol{G}}}
\newcommand{\bfR}{{\boldsymbol{R}}}
\newcommand{\bfQ}{{\boldsymbol{Q}}}
\newcommand{\bfT}{{\boldsymbol{T}}}
\newcommand{\bfU}{{\boldsymbol{U}}}
\newcommand{\bfX}{{\boldsymbol{X}}}
\newcommand{\bfY}{{\boldsymbol{Y}}}
\newcommand{\bfZ}{{\boldsymbol{Z}}}

\newcommand{\bfeta}{{\boldsymbol{\eta}}}
\newcommand{\bfxi}{{\boldsymbol{\xi}}}
\newcommand{\bfrho}{{\boldsymbol{\rho}}}

\newcommand{\PreP}{{\mathrm{PrePer}}}

\def\fM{{\mathfrak M}}
\def\fE{{\mathfrak E}}
\def\fF{{\mathfrak F}}



\newfont{\teneufm}{eufm10}
\newfont{\seveneufm}{eufm7}
\newfont{\fiveeufm}{eufm5}
%
%
\newfam\eufmfam
                \textfont\eufmfam=\teneufm \scriptfont\eufmfam=\seveneufm
                \scriptscriptfont\eufmfam=\fiveeufm
%
%
\def\frak#1{{\fam\eufmfam\relax#1}}
%

\def\ts{\thinspace}

\newtheorem{theorem}{Theorem}
\newtheorem{lemma}[theorem]{Lemma}
\newtheorem{claim}[theorem]{Claim}
\newtheorem{coro}[theorem]{Corollary}
\newtheorem{prop}[theorem]{Proposition}
\newtheorem{question}[theorem]{Open Question}

\newtheorem{rem}[theorem]{Remark}
\newtheorem{definition}[theorem]{Definition}

\newenvironment{dedication}
  {
   \thispagestyle{empty}
   \vspace*{\stretch{1}}
   \itshape             
   \raggedleft          
  }
  {\par 
   \vspace{\stretch{3}} 
  }

\numberwithin{table}{section}
\numberwithin{equation}{section}
\numberwithin{figure}{section}
\numberwithin{theorem}{section}


\def\squareforqed{\hbox{\rlap{$\sqcap$}$\sqcup$}}
\def\qed{\ifmmode\squareforqed\else{\unskip\nobreak\hfil
\penalty50\hskip1em\null\nobreak\hfil\squareforqed
\parfillskip=0pt\finalhyphendemerits=0\endgraf}\fi}

\def\fA{{\mathfrak A}}
\def\fB{{\mathfrak B}}

\def\cA{{\mathcal A}}
\def\cB{{\mathcal B}}
\def\cC{{\mathcal C}}
\def\cD{{\mathcal D}}
\def\cE{{\mathcal E}}
\def\cF{{\mathcal F}}
\def\cG{{\mathcal G}}
\def\cH{{\mathcal H}}
\def\cI{{\mathcal I}}
\def\cJ{{\mathcal J}}
\def\cK{{\mathcal K}}
\def\cL{{\mathcal L}}
\def\cM{{\mathcal M}}
\def\cN{{\mathcal N}}
\def\cO{{\mathcal O}}
\def\cP{{\mathcal P}}
\def\cQ{{\mathcal Q}}
\def\cR{{\mathcal R}}
\def\cS{{\mathcal S}}
\def\cT{{\mathcal T}}
\def\cU{{\mathcal U}}
\def\cV{{\mathcal V}}
\def\cW{{\mathcal W}}
\def\cX{{\mathcal X}}
\def\cY{{\mathcal Y}}
\def\cZ{{\mathcal Z}}

\def\nrp#1{\left\|#1\right\|_p}
\def\nrq#1{\left\|#1\right\|_m}
\def\nrqk#1{\left\|#1\right\|_{m_k}}
\def\Ln#1{\mbox{\rm {Ln}}\,#1}
\def\nd{\hspace{-1.2mm}}
\def\ord{{\mathrm{ord}}}
\def\Cc{{\mathrm C}}
\def\Pb{\,{\mathbf P}}

\def\va{{\mathbf{a}}}

\newcommand{\commB}[1]{\marginpar{%
\begin{color}{red}
\vskip-\baselineskip 
\raggedright\footnotesize
\itshape\hrule \smallskip F: #1\par\smallskip\hrule\end{color}}}

\newcommand{\commO}[1]{\marginpar{%
\begin{color}{cyan}
\vskip-\baselineskip 
\raggedright\footnotesize
\itshape\hrule \smallskip O: #1\par\smallskip\hrule\end{color}}}




\newcommand{\ignore}[1]{}

\def\vec#1{\boldsymbol{#1}}

\def\e{\mathbf{e}}



\def\GL{\mathrm{GL}}

\hyphenation{re-pub-lished}

\def\rank{{\mathrm{rk}\,}}
\def\dd{{\mathrm{dyndeg}\,}}
\def\lcm{{\mathrm{lcm}\,}}

\def\A{\mathbb{A}}
\def\B{\mathbf{B}}
\def \C{\mathbb{C}}
\def \F{\mathbb{F}}
\def \K{\mathbb{K}}
\def \L{\mathbb{L}}
\def \Z{\mathbb{Z}}
\def \P{\mathbb{P}}
\def \R{\mathbb{R}}
\def \Q{\mathbb{Q}}
\def \N{\mathbb{N}}
\def \S{\mathbb{S}}
\def \Z{\mathbb{Z}}

\def \nd{{\, | \hspace{-1.5 mm}/\,}}

\def\mand{\qquad\mbox{and}\qquad}

\def\Zn{\Z_n}

\def\Fp{\F_p}
\def\Fq{\F_q}
\def \fp{\Fp^*}
\def\\{\cr}
\def\({\left(}
\def\){\right)}
\def\fl#1{\left\lfloor#1\right\rfloor}
\def\rf#1{\left\lceil#1\right\rceil}
\def\vh{\mathbf{h}}
\def\ov#1{{\overline{#1}}}
\def\un#1{{\underline{#1}}}
\def\wh#1{{\widehat{#1}}}
\def\wt#1{{\widetilde{#1}}}
\newcommand{\abs}[1]{\left| #1 \right|}

\def\ZK{\Z_\K}
\def\LH{\cL_H}

\def \fI{\mathfrak{I}}
\def \fJ{\mathfrak{J}}
\def \fV{\mathfrak{V}}

\title[identities in number theory]
{On some identities in multiplicative number theory}

\author[O. Bordell\`{e}s]{Olivier Bordell\`{e}s}
\address{O.B.: 2 All\'ee de la combe, 43000 Aiguilhe, France}
\email{borde43@wanadoo.fr}

\author[B. Cloitre]{Benoit Cloitre}

\address{B.C. : 19 rue Louise Michel, 92300 Levallois-Perret, France}
\email{benoit7848c@yahoo.fr}



\date{}

\begin{abstract} Using elementary means, we prove several identities involving the M\"{o}bius function, generalizing in the multidimensional case well-known formulas coming from convolution arguments.
\end{abstract}

\maketitle

\section{Presentation of the results}

Our first result is a multidimensional generalization of the well-known convolution identity
$$\sum_{n \leqslant x} \mu(n) \left \lfloor \frac{x}{n} \right \rfloor = 1.$$

\begin{theorem}
\label{th:id1}
Let $r \in \Z_{\geqslant 2}$. For any real number $x \geqslant 1$
$$\sum_{n_1 \leqslant x, \dotsc, n_r \leqslant x} \mu \left( n_1 \dotsb n_r \right) \left \lfloor \frac{x}{n_1 \dotsb n_r} \right \rfloor = \sum_{n \leqslant x} (1-r)^{\omega(n)}.$$
\end{theorem}

Usual bounds in analytic number theory (see Lemma~\ref{lem:unitary} below) lead to the following estimates.

\begin{coro}
\label{cor:id0}
Let $r \in \Z_{\geqslant 2}$. There exists an absolute constant $c_0 >0$ and a constant $c_r \geqslant 1$, depending on $r$, such that, for any $x \geqslant c_r$ sufficiently large
$$\sum_{n_1 \leqslant x, \dotsc, n_r \leqslant x} \mu \left( n_1 \dotsb n_r \right) \left \lfloor \frac{x}{n_1 \dotsb n_r} \right \rfloor \ll_r x e^{-c_0 (\log x)^{3/5} (\log \log x)^{-1/5}}.$$
Furthermore, the Riemann Hypothesis is true if and only if, for any $\varepsilon > 0$
$$\sum_{n_1 \leqslant x, \dotsc, n_r \leqslant x} \mu \left( n_1 \dotsb n_r \right) \left \lfloor \frac{x}{n_1 \dotsb n_r} \right \rfloor \ll_{r,\varepsilon} x^{1/2+\varepsilon}.$$
\end{coro}

By a combinatorial argument, we obtain the following asymptotic formula.

\begin{coro}
\label{cor:id1}
Let $r \in \Z_{\geqslant 2}$. There exists an absolute constant $c_0 >0$ and a constant $c_r \geqslant 1$, depending on $r$, such that, for any $x \geqslant c_r$ sufficiently large
$$\sum_{1 \leqslant n_1 < \dotsb < n_r \leqslant x} \mu \left( n_1 \dotsb n_r \right) \left \lfloor \frac{x}{n_1 \dotsb n_r} \right \rfloor = \frac{(-1)^{r-1} x}{r(r-2)!} + O_r \left( x e^{-c_0 (\log x)^{3/5} (\log \log x)^{1/5}} \right).$$
Furthermore, the Riemann Hypothesis is true if and only if, for any $\varepsilon > 0$
$$\sum_{1 \leqslant n_1 < \dotsb < n_r \leqslant x} \mu \left( n_1 \dotsb n_r \right) \left \lfloor \frac{x}{n_1 \dotsb n_r} \right \rfloor = \frac{(-1)^{r-1} x}{r(r-2)!} + O_{r,\varepsilon} \left( x^{1/2+\varepsilon} \right).$$
\end{coro}

\bigskip

Our second identity is the analogue of  Theorem~\ref{th:id1} with $\mu$ replaced by $\mu^2$.

\begin{theorem}
\label{th:id15}
Let $r \in \Z_{\geqslant 2}$. For any real number $x \geqslant 1$
$$\sum_{n_1 \leqslant x, \dotsc, n_r \leqslant x} \mu \left( n_1 \dotsb n_r \right)^2 \left \lfloor \frac{x}{n_1 \dotsb n_r} \right \rfloor = \sum_{n \leqslant x} (1+r)^{\omega(n)}.$$
\end{theorem}

As for an asymptotic formula, we derive the following estimate from the contour integration method applied to the function $k^\omega$ (see \cite[Exercise~II.4.1]{ten} for instance).

\begin{coro}
\label{cor:id152}
Let $r \in \Z_{\geqslant 2}$ and $\varepsilon > 0$. For any large real number $x \geqslant 1$
$$\sum_{n_1 \leqslant x, \dotsc, n_r \leqslant x} \mu \left( n_1 \dotsb n_r \right)^2 \left \lfloor \frac{x}{n_1 \dotsb n_r} \right \rfloor = x \mathcal{P}_r(\log x) + O \left( x^{1 - \frac{3}{r+7} + \varepsilon} \right)$$
where $\mathcal{P}_r$ is a polynomial of degree $r$ and leading coefficient
$$\frac{1}{r!} \prod_p \left( 1 - \frac{1}{p} \right)^{r+1} \left( 1 + \frac{r+1}{p-1} \right).$$
\end{coro}

When $r=3$, we can use a recent result of \cite{zha} which allows us to improve on Corollary~\ref{cor:id152}. 
\begin{equation}
   \sum_{n_1 , n_2 , n_3 \leqslant x} \mu \left( n_1 n_2 n_3 \right)^2 \left \lfloor \frac{x}{n_1 n_2 n_3} \right \rfloor = x \mathcal{P}_3(\log x) + O \left( x^{1/2} (\log x)^5 \right) \label{eq:ex1}
\end{equation}
where $\mathcal{P}_3$ is a polynomial of degree $3$.

\bigskip

Our third result is quite similar to Theorem~\ref{th:id1}, but is simpler and sheds a new light onto the Piltz-Dirichlet divisor problem.

\begin{theorem}
\label{th:id2}
Let $r \in \Z_{\geqslant 2}$. For any real number $x \geqslant 1$
$$\sum_{n_1 \leqslant x, \dotsc, n_r \leqslant x} \left( \mu \left( n_1 \right) + \dotsb + \mu \left( n_r \right) \right) \left \lfloor \frac{x}{n_1 \dotsb n_r} \right \rfloor = r \sum_{n \leqslant x} \tau_{r-1}(n).$$
\end{theorem}

Known results from the the Piltz-Dirichlet divisor problem yield the following corollary (see \cite{bouw,ivio,kol,kolp} for the estimates of the remainder term below).

\begin{coro}
\label{cor:id2}
Let $r \in \Z_{\geqslant 3}$ and $\varepsilon >0 $. For any real number $x \geqslant 1$ sufficiently large
$$\sum_{n_1 \leqslant x, \dotsc, n_r \leqslant x} \left( \mu \left( n_1 \right) + \dotsb + \mu \left( n_r \right) \right) \left \lfloor \frac{x}{n_1 \dotsb n_r} \right \rfloor = r \underset{s=1}{\res} \left( s \zeta(s)^{r-1} x^s \right) + O_{r,\varepsilon} \left( x^{\alpha_r + \varepsilon} \right) $$
where 
\begin{center}
\begin{tabular}{cccccc}
$r$ & $3$ & $4$ & $10$ & $\in \left[ 121,161 \right]$ & $ \geqslant 161$ \\
& & & & &  \\
\hline
& & & & &  \\
$\alpha_r$ & $\frac{517}{\np{1648}}$ & $\frac{43}{\np{96}}$ & $\frac{35}{\np{54}}$ & $1 - \frac{1}{3} \left( \frac{2}{\np{4.45} (r-1)} \right)^{2/3}$ & $1 - \left( \frac{2}{\np{13.35} (r-\np{160.9})} \right)^{2/3}$ \\
\end{tabular}
\end{center}
\end{coro}

Once again, the case $r=3$ is certainly one of the most interesting one. Corollary~\ref{cor:id2} yields
\begin{equation}
   \sum_{n_1 , n_2 , n_3 \leqslant x} \left( \mu \left( n_1 \right) + \mu \left( n_2 \right) + \mu \left( n_3 \right)\right) \left \lfloor \frac{x}{n_1 n_2 n_3} \right \rfloor = 3x \log x + \left( 6 \gamma - 3 \right) x + O_\varepsilon \left( x^{\frac{517}{\np{1648}}+ \varepsilon} \right). \label{eq:ex2}
\end{equation}

\bigskip

Our last identity generalizes the well-known relation
\begin{equation}
   \sum_{d \mid n} \frac{\mu(d) \log d}{d} = - \frac{\varphi(n)}{n} \sum_{p \mid n} \frac{\log p}{p-1} \label{eq:phi}
\end{equation}
which can be proved in the following way:
\begin{eqnarray*}
   - \frac{\varphi(n)}{n} \sum_{p \mid n} \frac{\log p}{p-1} &=& - \frac{1}{n} \sum_{p \mid n} \varphi \left( \frac{n}{p} \right) \log p \\
   &=& - \frac{1}{n} \left( \Lambda \star \varphi \right) (n) \\
   &=& - \frac{1}{n} \left( - \mu \log \star \mathbf{1} \star \mu \star \id \right) (n) \\
   &=& \frac{1}{n} \left( \mu \log \star \id \right) (n) \\
   &=& \sum_{d \mid n} \frac{\mu(d) \log d}{d}.
\end{eqnarray*}

\begin{theorem}
\label{th:id3}
Let $n \in \Z_{\geqslant 1}$, $e \in \{1,2\}$ and $f$ be a multiplicative function such that $f(p) \neq (-1)^{e+1}$. Then
$$\sum_{d \mid n} \mu(d)^e f(d) \log d =  \prod_{p \mid n} \left( 1 + (-1)^e f(p) \right) \sum_{p \mid n} \frac{f(p) \log p}{f(p) + (-1)^e}.$$
\end{theorem}

Taking $f= \id^{-1}$ yields \eqref{eq:phi}, but many other consequences may be established with this result. We give some of them below.

\begin{coro}
\label{cor}
Let $k,n \in \Z_{\geqslant 1}$ and $e \in \{1,2\}$. Then
   \begin{eqnarray*}
     & & \sum_{d \mid n} \frac{\mu(d) \log d}{d^k} = -\frac{J_k(n)}{n^k} \sum_{p \mid n} \frac{\log p}{p^k-1}. \\
     & & \sum_{d \mid n} \frac{\mu(d)^2 \log d}{d^k} = \frac{\Psi_k(n)}{n^k} \sum_{p \mid n} \frac{\log p}{p^k+1}. \\
     & & \sum_{d \mid n} \frac{\mu(d)^2 \log d}{\varphi(d)} = \frac{n}{\varphi(n)} \sum_{p\mid n} \frac{\log p}{p}. \\
     & & \sum_{d \mid n} \frac{\mu(d) \log d}{\tau(d)} = - 2^{-\omega(n)} \log \gamma(n). \\
     & & \sum_{d \mid n} \mu(d)^e \tau_k(d) \log d = \left( 1 + (-1)^e k \right)^{\omega(n)} \times \frac{k \log \gamma(n)}{k+(-1)^e} \quad \left( k \in \Z_{\geqslant 2} \right). \\
     & & \sum_{d \mid n} \mu(d) \sigma(d) \log d = (-1)^{\omega(n)} \gamma(n) \left( \log \gamma(n) + \sum_{p\mid n} \frac{\log p}{p} \right).
   \end{eqnarray*}
\end{coro}

Note that a similar result has been proved in \cite[Theorems~3 and~5]{wak} in which the completely additive function $\log$ is replaced by the strongly additive function $\omega$. However, let us stress that the methods of proofs are completely different.

\section{Notation}

\noindent
\begin{scriptsize} $\triangleright$ \end{scriptsize} We use some classical multiplicative functions such as $\mu$, the M\"{o}bius function, $\varphi$, $\Psi$, $J_k$ and $\Psi_k$, respectively the Euler, Dedekind, $k$-th Jordan and $k$-th Dedekind totients. Recall that, for any $k \in \Z_{\geqslant 1}$
$$J_k (n) := n^k \prod_{p \mid n} \left( 1 - \frac{1}{p^k} \right) \quad \textrm{and} \quad \Psi_k (n) := n^k \prod_{p \mid n} \left( 1 + \frac{1}{p^k} \right).$$
Also, $\varphi = J_1$ and $\Psi = \Psi_1$. Next, $\gamma(n)$ is the \textit{squarefree kernel} of $n$, defined by
$$\gamma(n) := \prod_{p \mid n} p.$$

\medskip

\noindent
\begin{scriptsize} $\triangleright$ \end{scriptsize} Let $q \in \Z_{\geqslant 1}$. The notation $n \mid q^\infty$ means that every prime factor of $n$ is a prime factor of $q$. We define $\mathbf{1}_q^\infty$ to be the characteristic function of the integers $n$ satisfying $n \mid q^\infty$. It is important to note that
\begin{equation}
   \mathbf{1}_{q}^{\infty} (n) = \sum_{\substack{d \mid n \\ (d,q)=1}} \mu(d). \label{eq:infty}
\end{equation}
This can easily be checked for prime powers $p^\alpha$ and extended to all integers using multiplicativity.

\medskip

\noindent
\begin{scriptsize} $\triangleright$ \end{scriptsize} Finally, if $F$ and $G$ are two arithmetic function, the Dirichlet convolution product $F \star G$ is given by
$$(F \star G)(n) := \sum_{d \mid n} F(d) G(n/d).$$

\medskip

\noindent
We always use the convention that an empty product is equal to $1$.

\section{Proof of Theorem~\ref{th:id1}}

\subsection{Lemmas}

\begin{lemma}
\label{le2}
Let $q \in \Z_{\geqslant 1}$. For any $x \in \R_{\geqslant 1}$
$$\sum_{\substack{n \leqslant x \\ (n,q)=1}} \mu(n) \left \lfloor \frac{x}{n} \right \rfloor = \sum_{\substack{n \leqslant x \\ n \mid q^\infty}} 1.$$
\end{lemma}

\begin{proof}
This follows from \eqref{eq:infty} and the convolution identity
$$\sum_{\substack{n \leqslant x \\ (n,q)=1}} \mu(n) \left \lfloor \frac{x}{n} \right \rfloor = \sum_{n \leqslant x} \sum_{\substack{d \mid n \\ (d,q)=1}} \mu(d)$$
as asserted.
\end{proof}

Note that the sum of the left-hand side has also been investigated in \cite{gup} by a completely different method. 

\begin{lemma}
\label{le5}
Let $q \in \Z_{\geqslant 1}$. For any $x \in \R_{\geqslant 1}$
$$\sum_{\substack{n \leqslant x \\ q \mid n}} \mu(n) \sum_{\substack{k \leqslant x/n \\ k \mid n^\infty}} 1 = \sum_{\substack{n \leqslant x \\ q \mid \gamma(n)}} (-1)^{\omega(n)}.$$
\end{lemma}

\begin{proof}
We first prove that, for any $n \in \Z_{\geqslant 1}$ and any squarefree divisor $d$ of $n$
\begin{equation}
   \mathbf{1}_d^\infty \left( \tfrac{n}{d} \right) = \begin{cases} 1, & \textrm{if\ } d = \gamma(n) \\ 0, & \textrm{otherwise.} \end{cases} \label{eq:id3}
\end{equation}
The result is obvious if $n=1$. Assume $n \geqslant 2$ and let $d$ be a squarefree divisor of $n$. Then, from \eqref{eq:infty}
$$\mathbf{1}_d^\infty \left( \tfrac{n}{d} \right) = \sum_{\substack{\delta \mid \frac{n}{d} \\ (\delta,d)=1}} \mu(\delta) = \sum_{\delta \mid \frac{\gamma(n)}{d}} \mu(\delta) $$
implying \eqref{eq:id3}. Now, for any $n \in \Z_{\geqslant 1}$
$$\sum_{\substack{d \mid n \\ (n/d) \mid d^\infty \\ q \mid d}} \mu(d) = \sum_{\substack{d \mid n \\ d = \gamma(n) \\ q \mid d}} \mu(d) = \begin{cases} \mu \left( \gamma(n) \right), & \textrm{if\ } q \mid \gamma(n) \\ 0, & \textrm{otherwise} \end{cases} = \begin{cases} (-1)^{\omega(n)}, & \textrm{if\ } q \mid \gamma(n) \\ 0, & \textrm{otherwise.} \end{cases}$$
The asserted result then follows from
$$\sum_{\substack{n \leqslant x \\ q \mid \gamma(n)}} (-1)^{\omega(n)} = \sum_{n \leqslant x} \sum_{\substack{d \mid n \\ (n/d) \mid d^\infty \\ q \mid d}} \mu(d) = \sum_{\substack{d \leqslant x \\ q \mid d}} \mu(d) \sum_{\substack{k \leqslant x/d \\ k \mid d^\infty}} 1$$
as required.
\end{proof}

\begin{lemma}
\label{le6}
For any $k \in \Q^*$ and $n \in \Z_{\geqslant 1}$
$$\sum_{d \mid n} \mu(d)^2 k^{\omega(d)} = (k+1)^{\omega(n)}.$$
\end{lemma}

\begin{proof}
This is well-known. For instance, this can be checked for prime powers and then extended to all integers by multiplicativity.
\end{proof}

\begin{lemma}
\label{le7}
Let $q \in \Z_{\geqslant 1}$ squarefree. For any $k \in \Z_{\geqslant 1}$
$$\sum_{n_1 \dotsb n_k \mid q} \mu \left( n_1 \right)^2 \dotsb \mu \left( n_k \right)^2 = (k+1)^{\omega(q)}.$$
\end{lemma}

\begin{proof}
Let $S_k(q)$ be the sum of the left-hand side. We proceed by induction on $k$, the case $k=1$ being Lemma~\ref{le6} with $k=1$. Suppose that the result is true for some $k \geqslant 1$. Then, using induction hypothesis and Lemma~\ref{le6}, we get
\begin{eqnarray*}
   S_{k+1}(q) &=& \sum_{n \mid q} \mu \left( n \right)^2 S_k \left( \frac{q}{n} \right) = \sum_{n \mid q} \mu \left( n \right)^2 \left( k+1 \right)^{\omega(q/n)} \\
   &=& (k+1)^{\omega(q)}\sum_{n \mid q} \mu \left( n \right)^2 \left( k+1 \right)^{-\omega \left( n \right) } \\
   &=& (k+1)^{\omega(q)} \left( \tfrac{1}{k+1} + 1 \right)^{\omega(q)} = (k+2)^{\omega(q)}
\end{eqnarray*}
completing the proof.
\end{proof}

\subsection{Proof of Theorem~\ref{th:id1}}

\noindent
From Lemma~\ref{le2}, we have
\begin{eqnarray*}
   & & \sum_{n_1 \leqslant x, \dotsc, n_r \leqslant x} \mu \left( n_1 \dotsb n_r \right) \left \lfloor \frac{x}{n_1 \dotsb n_r} \right \rfloor \\
   &=&  \sum_{n_1 \leqslant x, \dotsc, n_{r-1} \leqslant x} \mu \left( n_1 \dotsb n_{r-1} \right) \sum_{\substack{n_r \leqslant x/\left( n_1 \dotsb n_{r-1} \right) \\ \left( n_r, n_1 \dotsb n_{r-1} \right) = 1}} \mu \left( n_r \right) \left \lfloor \frac{x/\left( n_1 \dotsb n_{r-1} \right)}{n_r} \right \rfloor \\
   &=&  \sum_{n_1 \leqslant x, \dotsc, n_{r-1} \leqslant x} \mu \left( n_1 \dotsb n_{r-1} \right) \sum_{\substack{n_r \leqslant x/\left( n_1 \dotsb n_{r-1} \right) \\ n_r \mid \left( n_1 \dotsb n_{r-1} \right)^\infty}} 1.
\end{eqnarray*}

\noindent
The change of variable $m=n_1 \dotsb n_{r-1}$ yields
\begin{eqnarray*}
   \sum_{n_1 \leqslant x, \dotsc, n_r \leqslant x} \mu \left( n_1 \dotsb n_r \right) \left \lfloor \frac{x}{n_1 \dotsb n_r} \right \rfloor &=& \sum_{n_1 \leqslant x, \dotsc, n_{r-2} \leqslant x} \sum_{\substack{m \leqslant x \\ n_1 \dotsb n_{r-2} \mid m}} \mu(m) \sum_{\substack{n_r \leqslant x/m \\ n_r \mid m^\infty}} 1 \\
   &=& \sum_{n_1 \leqslant x, \dotsc, n_{r-2} \leqslant x} \sum_{\substack{m \leqslant x \\ n_1 \dotsb n_{r-2} \mid \gamma(m)}} (-1)^{\omega(m)}
\end{eqnarray*} 
where we used Lemma~\ref{le5} with $q=n_1 \dotsb n_{r-2}$. Now since $n_1 \dotsb n_{r-2} \mid \gamma(m)$ is equivalent to both $n_1 \dotsb n_{r-2} \mid \gamma(m)$ and $\mu \left (n_1 \right )^2 = \dotsb = \mu \left (n_{r-2} \right )^2 = 1$, we infer
\begin{eqnarray*}
   \sum_{n_1 \leqslant x, \dotsc, n_r \leqslant x} \mu \left( n_1 \dotsb n_r \right) \left \lfloor \frac{x}{n_1 \dotsb n_r} \right \rfloor &=& \sum_{m \leqslant x} (-1)^{\omega(m)} \sum_{\substack{n_1 \leqslant x, \dotsc, n_{r-2} \leqslant x \\ n_1 \dotsb n_{r-2} \mid \gamma(m)}} \mu \left (n_1 \right )^2 \dotsb  \mu \left (n_{r-2} \right )^2 \\
   &=& \sum_{m \leqslant x} (-1)^{\omega(m)} \sum_{n_1 \dotsb n_{r-2} \mid \gamma(m)} \mu \left (n_1 \right )^2 \dotsb  \mu \left (n_{r-2} \right )^2.
\end{eqnarray*} 
Lemma~\ref{le7} then gives
\begin{eqnarray*}
   \sum_{n_1 \leqslant x, \dotsc, n_r \leqslant x} \mu \left( n_1 \dotsb n_r \right) \left \lfloor \frac{x}{n_1 \dotsb n_r} \right \rfloor &=& \sum_{m \leqslant x} (-1)^{\omega(m)} (r-1)^{\omega \left( \gamma(m) \right)} \\
   &=& \sum_{m \leqslant x} (-1)^{\omega(m)} (r-1)^{\omega (m)} = \sum_{m \leqslant x} (1-r)^{\omega(m)}
\end{eqnarray*} 
as asserted.
\qed

\subsection{Proof of Corollaries~\ref{cor:id0} and~\ref{cor:id1}}

We start with an estimate which is certainly well-known, but we provide a proof for the sake of completeness. The unconditional estimate is quite similar to the usual bound given by the Selberg-Delange method (see \cite[Th\'{e}or\`{e}me~II.6.1]{ten} with $z \in \Z_{< 0}$).

\begin{lemma}
\label{lem:unitary}
Let $k \in \Z_{\geqslant 1}$. There exists an absolute constant $c_0 >0$ and a constant $c_k \geqslant 1$, depending on $k$, such that, for any $x \geqslant c_k$ sufficiently large
$$\sum_{n  \leqslant x} (-k)^{\omega(n)} \ll_k x e^{-c_0 (\log x)^{3/5} (\log \log x)^{-1/5}}.$$
Furthermore, the Riemann Hypothesis is true if and only if, for any $\varepsilon > 0$
$$\sum_{n  \leqslant x} (-k)^{\omega(n)} \ll_{k,\varepsilon} x^{1/2+\varepsilon}.$$
\end{lemma}

\begin{proof}
Set $f_k:=(-k)^\omega$. If $L(s,f_k)$ is the Dirichlet series of $f_k$, then usual computations show that, for any $s = \sigma + it \in \C$ such that $\sigma > 1$
$$L(s,f_k) = \zeta(s)^{-k} \zeta(2s)^{-\frac{1}{2}k(k+1)} G_k(s)$$
where $G_k(s)$ is a Dirichlet series absolutely convergent in the half-plane $\sigma > \frac{1}{3}$. Set $c := \np{57.54}^{-1}$, $\kappa := 1 + \frac{1}{\log x}$,
$$T := e^{c (\log x)^{3/5}(\log \log x)^{-1/5}} \ \textrm{and} \ \alpha := \alpha(T) = c (\log T)^{-2/3} (\log \log T)^{-1/3}.$$
We use Perron's summation formula in the shape \cite[Corollary~2.2]{liu}, giving
\begin{eqnarray*}
   \sum_{n  \leqslant x} (-k)^{\omega(n)} &=& \frac{1}{2 \pi i} \int_{\kappa - iT}^{\kappa + iT} \frac{G_k(s)}{\zeta(s)^{k} \zeta(2s)^{\frac{1}{2}k(k+1)}} \frac{x^s}{s} \, \textrm{d}s \\
   & & {} + O \left( \sum_{x-x/\sqrt{T} < n \leqslant x+x/\sqrt{T}} \left |f_k(n) \right| + \frac{x^\kappa}{\sqrt{T}} \sum_{n=1}^\infty \frac{\left |f_k(n) \right|}{n^\kappa} \right).
\end{eqnarray*}
By \cite{for}, $\zeta(s)$ has no zero in the region $\sigma \geqslant 1 - c (\log |t|)^{-2/3} (\log \log |t|)^{-1/3}$ and $|t| \geqslant 3$, so that we may shift the line of integration to the left and apply Cauchy's theorem in the rectangle with vertices $\kappa \pm iT$, $1-\alpha \pm iT$. In this region
$$\zeta(s)^{-1} \ll \left (\log (|t|+2) \right )^{2/3} \left( \log \log (|t|+2) \right)^{1/3}.$$
Therefore, the contribution of the horizontal sides does not exceed
$$\ll xT^{-1} (\log T)^{2k/3} (\log \log T)^{k/3} $$
and the contribution of the vertical side is bounded by
$$\ll x^{1- \alpha} (\log T)^{1+2k/3} (\log \log T)^{k/3}.$$
Since $T \ll x^{1-\varepsilon}$, Shiu's theorem \cite{shi} yields
$$\sum_{x-x/\sqrt{T} < n \leqslant x+x/\sqrt{T}} \left |f_k(n) \right| \leqslant \sum_{x-x/\sqrt{T} < n \leqslant x+x/\sqrt{T}} \tau_k (n) \ll \frac{x}{\sqrt{T}} (\log x)^{k-1}.$$
With the choice of $\kappa$, the $2$nd error term does not exceed
$$\leqslant \frac{x^\kappa}{\sqrt{T}} \sum_{n=1}^\infty \frac{\tau_k (n)}{n^\kappa} = \frac{x^\kappa}{\sqrt{T}} \zeta(\kappa)^k \ll \frac{x}{\sqrt{T}} (\log x)^k.$$
Since the path of integration does not surround the origin, nor the poles of the integrand, Cauchy's theorem and the choice of $T$ give the asserted estimate for any $c_0 \leqslant \frac{1}{4}c$ and any real number $x$ satisfying $x \geqslant \exp \left( c_1 k^{10/3} \right)$, say, where $c_1 \geqslant 1$ is absolute. 
\item[]
Now let $x,T \geqslant 2$, with $x$ large and $T \leqslant x^2$. If the Riemann Hypothesis is true, then by Perron's formula again
$$\sum_{n  \leqslant x} (-k)^{\omega(n)} = \frac{1}{2 \pi i} \int_{2 - iT}^{2 + iT} \frac{G_k(s)}{\zeta(s)^{k} \zeta(2s)^{\frac{1}{2}k(k+1)}} \frac{x^s}{s} \, \textrm{d}s + O_{k,\varepsilon} \left( \frac{x^{2+\varepsilon}}{T} \right).$$
We shift the line $\sigma = 2$ to the line $\sigma = \frac{1}{2} + \varepsilon$. In the rectangle with vertices $2 \pm iT$, $\frac{1}{2} + \varepsilon \pm iT$, the Riemann Hypothesis implies that $\zeta(s)^{-1} \ll |t|^{\varepsilon/k}$, so that similar argument as above yields
$$\sum_{n  \leqslant x} (-k)^{\omega(n)} \ll_{k,\varepsilon} x^\varepsilon \left( x^2 T^{-1} + x^{1/2} \right)$$
and the choice of $T=x^2$ gives the asserted estimate. On the other hand, if 
$$\sum_{n  \leqslant x} (-k)^{\omega(n)} \ll_{k,\varepsilon} x^{1/2+\varepsilon}$$
then the series $L(s,f_k)$ converges for $\sigma > \frac{1}{2}$, and then defines an analytic function in this half-plane. Hence $\zeta(s)$ does not vanish in this region, and the Riemann Hypothesis holds. 
\end{proof}

Corollary~\ref{cor:id0} is a direct consequence of Theorem~\ref{th:id1} and Lemma~\ref{lem:unitary}. The proof of Corollary~\ref{cor:id1} uses the following combinatorial identity.

\begin{lemma}
\label{lem:combinatoric}
Let $r \in \Z_{\geqslant 2}$. For any $x \in \R_{\geqslant 1}$
\begin{eqnarray*}
   & & \sum_{1 \leqslant n_1 < \dotsb < n_r \leqslant x} \mu \left( n_1 \dotsb n_r \right) \left \lfloor \frac{x}{n_1 \dotsb n_r} \right \rfloor = \frac{1}{r!} \sum_{n_1 ,\dotsc ,n_r \leqslant x} \mu \left( n_1 \dotsb n_r \right) \left \lfloor \frac{x}{n_1 \dotsb n_r} \right \rfloor \\
   & & {} - \frac{1}{r!}\sum_{j=2}^{r-2} (-1)^j (j-1) {r \choose j} \sum_{n \leqslant x} (1-r+j)^{\omega(n)} - \frac{(-1)^r \lfloor x \rfloor}{r(r-2)!} + \frac{(-1)^{r}(r-2)}{(r-1)!} .
\end{eqnarray*}
\end{lemma}

\begin{proof}
Set $u\left( n_1 ,\dotsc ,n_r \right) := \mu \left( n_1 \dotsb n_r \right) \left \lfloor \frac{x}{n_1 \dotsb n_r} \right \rfloor$. Since $u$ is symmetric with respect to the $r$ variables $n_1 ,\dotsc ,n_r$, multiplying the left-hand side by $r!$ amounts to summing in the hypercube $\left[ 1,x \right]^r$, but we must take the diagonals into account. By a sieving argument, we get
\begin{eqnarray*}
   \sum_{n_1 ,\dotsc ,n_r \leqslant x}  u\left( n_1 ,\dotsc ,n_r \right) &=& r! \sum_{1 \leqslant n_1 < \dotsb < n_r \leqslant x} u\left( n_1 ,\dotsc ,n_r\right) \\
   & & {}+ \sum_{j=2}^r (-1)^{j} (j-1) {r \choose j} \sum_{n_j,\dotsc,n_r \leqslant x} u \left( n_j, \dotsc, n_j,n_{j+1}, \dotsc, n_r \right)
\end{eqnarray*}
where the variable $n_j$ appears $j$ times in the inner sum. Now since
$$u \left( n_j, \dotsc, n_j,n_{j+1}, \dotsc, n_r \right) = \mu \left( n_j^j n_{j+1} \dotsb n_r \right) \left \lfloor \frac{x}{n_j^j n_{j+1} \dotsb n_r} \right \rfloor$$
so that $u \left( n_j, \dotsc, n_j,n_{j+1}, \dotsc, n_r \right) = 0$ as soon as $n_j > 1$ since $j \geqslant 2$, we infer that the inner sum is
$$= \left\lbrace \begin{array}{rcll} \displaystyle \sum_{n_{j+1}, \dotsc, n_r \leqslant x} \mu \left( n_{j+1} \dotsb n_r \right) \left \lfloor \frac{x}{n_{j+1} \dotsb n_r} \right \rfloor &=& \displaystyle \sum_{n \leqslant x} (1-r+j)^{\omega(n)} & \textrm{if\ } 2 \leqslant j \leqslant r-2 \\ & & & \\ \displaystyle \sum_{n_{r-1}, n_r \leqslant x} \mu \left( n_{r-1}^{r-1} n_r \right) \left \lfloor \frac{x}{n_{r-1}^{r-1} n_r } \right \rfloor &=& 1 & \textrm{if\ } j=r-1 \\ & & & \\ \displaystyle \sum_{n_r \leqslant x} \mu \left( n_r^r \right) \left \lfloor \frac{x}{n_r^r } \right \rfloor &=& \lfloor x  \rfloor & \textrm{if\ } j=r
\end{array} \right.$$
where we used Theorem~\ref{th:id1} when $2 \leqslant j \leqslant r-2$, implying the asserted result.
\end{proof}

Corollary~\ref{cor:id1} follows immediately from Theorem~\ref{th:id1}, Lemmas~\ref{lem:unitary} and~\ref{lem:combinatoric}.
\qed

\section{Proof of Theorem~\ref{th:id15}}

\subsection{Lemmas}

\begin{lemma}
\label{le8}
Let $q \in \Z_{\geqslant 1}$. For any $x \in \R_{\geqslant 1}$
$$\sum_{\substack{n \leqslant x \\ (n,q)=1}} \mu(n)^2 \left \lfloor \frac{x}{n} \right \rfloor = \sum_{\substack{b \leqslant x \\ b \mid q^\infty}} \sum_{\substack{a \leqslant x/b \\ (a,q)=1}} 2^{\omega(a)}.$$
\end{lemma}

\begin{proof}
The proof is similar to that of Lemma~\ref{le2} except that we replace \eqref{eq:infty} by the convolution identity
$$\sum_{\substack{d \mid n \\ (d,q)=1}} \mu(d)^2 = 2^{\omega(a)}$$
where $n=ab$, $(a,b)=(a,q)=1$ and $b \mid q^\infty$.
\end{proof}

Our next result is the analogue of Lemma~\ref{le5}.

\begin{lemma}
\label{le9}
Let $q \in \Z_{\geqslant 1}$. For any $x \in \R_{\geqslant 1}$
$$\sum_{\substack{d \leqslant x \\ q \mid d}} \mu(d)^2 \sum_{\substack{k \leqslant x/d \\ k \mid d^\infty}} \ \sum_{\substack{h \leqslant x/(kd) \\ (h,d)=1}} 2 ^{\omega(h)} = \sum_{h \leqslant x} 2 ^{\omega(h)} \sum_{\substack{n \leqslant x/h \\ q \mid \gamma(n) \\ \left (h,\gamma(n) \right ) = 1}} 1.$$
\end{lemma}

\begin{proof}
The sum of the left-hand side is equal to
\begin{equation}
   \sum_{h \leqslant x} 2 ^{\omega(h)} \sum_{\substack{d \leqslant x/h \\ q \mid d \\ (d,h)=1}} \mu(d)^2 \sum_{\substack{k \leqslant x/(hd) \\ k \mid d^\infty}} 1. \label{eq:le9}
\end{equation}
Now as in the proof of Lemma~\ref{le5}, we derive
$$\sum_{\substack{d \mid n \\ (n/d) \mid d^\infty \\ q \mid d \\ (h,d) = 1}} \mu(d)^2 = \begin{cases} 1, & \textrm{if\ } q \mid \gamma(n) \ \textrm{and\ } \left( h, \gamma(n) \right) = 1 \\ 0, & \textrm{otherwise} \end{cases}$$
so that the inner sum of the right-hand side of the lemma is
$$\sum_{\substack{n \leqslant x/h \\ q \mid \gamma(n) \\ \left (h,\gamma(n) \right ) = 1}} 1 = \sum_{n \leqslant x/h} \sum_{\substack{d \mid n \\ (n/d) \mid d^\infty \\ q \mid d \\ (h,d) = 1}} \mu(d)^2 = \sum_{\substack{d \leqslant x/h \\ q \mid d \\(h,d)=1}} \mu(d)^2 \sum_{\substack{k \leqslant x/(hd) \\ k \mid d^\infty}} 1$$
and inserting this sum in \eqref{eq:le9} gives the asserted result.
\end{proof}

\begin{lemma}
\label{le10}
Let $r \in \Z_{\geqslant 1}$. For any $x \in \R_{\geqslant 1}$
$$\sum_{m \leqslant x} 2^{\omega(m)} \sum_{\substack{n \leqslant x/m \\ \left( m, \gamma(n) \right) = 1}} r^{\omega(n)} = \sum_{m \leqslant x} (r+2)^{\omega(m)}.$$
\end{lemma}

\begin{proof}
Define
$$\varphi_r(n) := \sum_{\substack{d \mid n \\ \left( d,\gamma(n/d) \right) = 1}} 2^{\omega(d)} r^{\omega(n/d)}$$
so that the sum of the left-hand side is
$$\sum_{n \leqslant x} \varphi_r(n).$$
The lemma follows by noticing that the function $\varphi_r$ is multiplicative and that $\varphi_r \left( p^\alpha \right) = r+2$ for all prime powers $p^\alpha$.
\end{proof}

\subsection{Proof of Theorem~\ref{th:id15}}

Let $S_r(x)$ be the sum of the theorem. By Lemma~\ref{le8} and the change of variable $m=n_1 \dotsb n_{r-1}$, we get
\begin{eqnarray*}
   S_r(x) &=& \sum_{n_1 \leqslant x, \dotsc, n_{r-1} \leqslant x} \mu \left( n_1 \dotsb n_{r-1} \right)^2 \sum_{\substack{n_r \leqslant x /(n_1 \dotsb n_{r-1}) \\ \left( n_r, n_1 \dotsb n_{r-1} \right) = 1}} \mu \left( n_r \right)^2 \left \lfloor \frac{x/(n_1 \dotsb n_{r-1})}{n_r} \right \rfloor \\
   &=& \sum_{n_1 \leqslant x, \dotsc, n_{r-1} \leqslant x} \mu \left( n_1 \dotsb n_{r-1} \right)^2 \sum_{\substack{n_r \leqslant x /(n_1 \dotsb n_{r-1}) \\ n_r \mid \left( n_1 \dotsb n_{r-1} \right)^\infty}} \ \sum_{\substack{n_{r+1} \leqslant x / (n_1 \dotsc n_r) \\ \left( n_{r+1}, n_1 \dotsb n_{r-1} \right) = 1}} 2^{\omega \left( n_{r+1} \right)} \\
   &=& \sum_{n_1 \leqslant x, \dotsc, n_{r-2} \leqslant x} \sum_{\substack{m \leqslant x \\ n_1 \dotsb n_{r-2} \mid m}} \mu(m)^2 \sum_{\substack{n_r \leqslant x/m \\ n_r \mid m^\infty}} \quad \sum_{\substack{n_{r+1} \leqslant x/(m n_r) \\ \left( n_{r+1}, m \right) = 1}} 2^{\omega \left( n_{r+1} \right)}.
\end{eqnarray*}
Now Lemma~\ref{le9} yields
\begin{eqnarray*}
   S_r(x) &=& \sum_{n_1 \leqslant x, \dotsc, n_{r-2} \leqslant x} \ \sum_{m \leqslant x} 2^{\omega(m)} \ \sum_{\substack{n \leqslant x/m \\ n_1 \dotsb n_{r-2} \mid \gamma(n) \\ \left( m, \gamma(n) \right) = 1}} 1 \\
   &=& \sum_{m \leqslant x} 2^{\omega(m)} \ \sum_{\substack{n \leqslant x/m \\ \left( m, \gamma(n) \right) = 1}} \ \sum_{\substack{n_1 \leqslant x, \dotsc, n_{r-2} \leqslant x \\ n_1 \dotsb n_{r-2} \mid \gamma(n)}} 1 \\
   &=& \sum_{m \leqslant x} 2^{\omega(m)} \ \sum_{\substack{n \leqslant x/m \\ \left( m, \gamma(n) \right) = 1}} \ \sum_{n_1 \dotsb n_{r-2} \mid \gamma(n)} \mu \left( n_1 \right)^2 \dotsb \mu \left( n_{r-2} \right)^2 \\
\end{eqnarray*}
and Lemmas~\ref{le7} and~\ref{le10} imply that
\begin{eqnarray*}
   S_r(x) &=& \sum_{m \leqslant x} 2^{\omega(m)} \ \sum_{\substack{n \leqslant x/m \\ \left( m, \gamma(n) \right) = 1}} (r-1)^{\omega \left( \gamma(n) \right)} \\
   &=& \sum_{m \leqslant x} 2^{\omega(m)} \ \sum_{\substack{n \leqslant x/m \\ \left( m, \gamma(n) \right) = 1}} (r-1)^{\omega \left( n \right)} \\
   &=& \sum_{m \leqslant x} (r+1)^{\omega(m)}
\end{eqnarray*}
as required.
\qed

\section{Proof of Theorem~\ref{th:id2}}

\noindent
This identity is a consequence of the following more general result.

\begin{theorem}
\label{th:csq-id2}
Let $f : \Z_{\geqslant 1} \longrightarrow \C$ be an arithmetic function and set
$$S_{r}(x)=\sum_{n_{1}\leqslant x,...,n_{r}\leqslant x}\left( f \left( n_1 \right) + \dotsb + f \left( n_r \right)  \right) \left \lfloor \frac{x}{n_1 \dotsb n_r}\right \rfloor \quad \left( x \geqslant 1 \right).$$
Let $(T_r(x))$ be the sequence recursively defined by
$$T_{1}(x)=\sum_{n \leqslant x} \left \lfloor \frac{x}{n} \right \rfloor \left(f \star \mathbf{1}\right) (n) \quad \text{and} \quad T_{r}(x) = \sum_{n \leqslant x} T_{r-1} \left( \frac{x}{n} \right) \quad \left( r \in \Z_{\geqslant 2} \right).$$
Then, for any $r \in \Z_{\geqslant 2}$ and $x \in \R_{\geqslant 1}$
$$S_{r}(x)=r T_{r-1}(x).$$
\end{theorem}

\begin{proof}
An easy induction shows that, for any $r \in \Z_{\geqslant 2}$
\begin{equation}
   T_{r-1}(x) = \sum_{n_1 \leqslant x} \sum_{n_2 \leqslant x/n_1} \dotsb \sum_{n_{r-2} \leqslant \frac{x}{n_1 \dotsb n_{r-3}}} T_1 \left( \frac{x}{n_1 \dotsb n_{r-2}} \right). \label{eq:id2}
\end{equation}
Now
\begin{eqnarray*}
   S_r(x) &=& r \sum_{n_1 \leqslant x} \sum_{n_2 \leqslant x/n_1} \dotsb \sum_{n_{r-2} \leqslant \frac{x}{n_1 \dotsb n_{r-3}}} \sum_{n_{r-1} \leqslant \frac{x}{n_1 \dotsb n_{r-2}}} f \left( n_{r-1} \right) \sum_{n_r \leqslant \frac{x}{n_1 \dotsb n_{r-1}}} \left \lfloor \frac{x}{n_1 \dotsb n_r} \right \rfloor \\
   &=& r \sum_{n_1 \leqslant x} \sum_{n_2 \leqslant x/n_1} \dotsb \sum_{n_{r-2} \leqslant \frac{x}{n_1 \dotsb n_{r-3}}} \sum_{n_{r-1} \leqslant \frac{x}{n_1 \dotsb n_{r-2}}} \left \lfloor \frac{x}{n_1 \dotsb n_{r-1}} \right \rfloor \sum_{n_r \mid n_{r-1}} f \left( n_{r} \right)  \\
   &=& r \sum_{n_1 \leqslant x} \sum_{n_2 \leqslant x/n_1} \dotsb \sum_{n_{r-2} \leqslant \frac{x}{n_1 \dotsb n_{r-3}}} \sum_{n_{r-1} \leqslant \frac{x}{n_1 \dotsb n_{r-2}}} \left \lfloor \frac{x}{n_1 \dotsb n_{r-1}} \right \rfloor \left( f \star \mathbf{1} \right)  \left( n_{r-1} \right) \\
   &=& r \sum_{n_1 \leqslant x} \sum_{n_2 \leqslant x/n_1} \dotsb \sum_{n_{r-2} \leqslant \frac{x}{n_1 \dotsb n_{r-3}}} T_1 \left( \frac{x}{n_1 \dotsb n_{r-2}} \right) \\
   &=& r T_{r-1}(x)
\end{eqnarray*}
by \eqref{eq:id2}, as required.
\end{proof}
When $f = \mu$, then $\mu \star \mathbf{1} = \delta$ where $\delta$ is the identity element of the Dirichlet convolution product, i.e. $\delta(n) = 1$ if $n=1$ and $\delta(n) = 0$ otherwise, and hence $T_r (x) = \sum_{n \leqslant x}\tau_r(n)$ by induction.

\section{Proof of Theorem~\ref{th:id3}}

\noindent
For any $s \in \R_{\geqslant 0}$, define
$$F(s) := \sum_{d \mid n} \frac{\mu(d)^e f(d)}{d^s} = \prod_{p \mid n} \left( 1 + \frac{(-1)^e f(p)}{p^s} \right).$$
Then
$$\sum_{d \mid n} \mu(d)^e f(d) \log d = - F^{\, \prime} (0)$$
with
$$- \frac{F^{\, \prime}}{F} (s) = \sum_{p \mid n} \frac{f(p) \log p}{f(p) + (-1)^e p^s}$$
and we complete the proof noticing that $F(0) = \prod_{p \mid n} \left( 1 + (-1)^e f(p) \right)$.
\qed

\end{document}